\documentclass[11pt,english,leqno]{article}
\usepackage[utf8]{inputenc}
\usepackage[T1]{fontenc}
\usepackage{babel}
 \usepackage{amsmath}
 \usepackage{mathtools}
\usepackage{amsfonts}
\usepackage{amssymb}
\usepackage{graphicx}
\usepackage{yhmath}
\usepackage{amsthm}
\usepackage{bbm}
\usepackage{mathtools}
\usepackage[colorlinks=true,
            linkcolor=blue,
            urlcolor=blue,
            citecolor=blue]{hyperref}
\usepackage{authblk}
\usepackage{geometry}
\usepackage{algorithm}
\usepackage[noend]{algpseudocode}
\usepackage[vcentermath]{youngtab}
\usepackage{pgfplots}
\pgfplotsset{compat=newest}
\usetikzlibrary{calc}
\usepackage{mathrsfs}
\definecolor {processblue}{cmyk}{0.96,0,0,0}

\author{
Mohamed Slim Kammoun
\footnote{
 Univ. Lille, CNRS, UMR 8524 - Laboratoire Paul Painlevé, F-59000 Lille, France.
 Email: \href{mailto:mohamed-slim.kammoun@univ-lille.fr}{\nolinkurl{mohamed-slim.kammoun@univ-lille.fr}}. 
}
\quad
Myl\`ene Ma\"ida
\footnote{
 Univ. Lille, CNRS, UMR 8524 - Laboratoire Paul Painlevé, F-59000 Lille, France.
 Email: \href{mailto:mylene.maida@univ-lille.fr}{\nolinkurl{mylene.maida@univ-lille.fr}}. 
}
}

\title{A product of  invariant random permutations has the same small cycle structure as uniform}
\newtheorem{theorem}{Theorem}
\newtheorem{corollary}[theorem]{Corollary}
\newtheorem{lemma}[theorem]{Lemma}
\newtheorem{proposition}[theorem]{Proposition}

\theoremstyle{definition}

\usepackage{tikz}
\usetikzlibrary{automata, positioning}
\usepackage{natbib}
\newcommand{\E}{\mathbb{E}}
\newcommand{\p}{\mathbb{P}}
\newcommand{\s}{\mathfrak{S}_n}

\newcommand{\bigslant}[2]{{\raisebox{.2em}{$#1$}\left/\raisebox{-.2em}{$#2$}\right.}}

\begin{document}
\maketitle
\begin{abstract} 
We use moment method to understand the cycle structure of the composition of independent invariant permutations. We prove that under a good control on fixed points and cycles of length $2$,  the  limiting   joint distribution of  the number of small cycles is the same as in the uniform case i.e. for any positive integer $k$,  the number of cycles of length $k$ converges to the Poisson distribution with parameter $\frac{1}{k}$ and is asymptotically independent of the number of cycles of length $k'\neq k$. 
\end{abstract}
\section{Introduction and main results} 
We denote % by $\m$, the set of bistochastic random permutations of size $n$, by $tr$ the trace operator,
by $\s$ the group of permutations of $\{1,\dots,n\},$  by $\#_k\,\sigma$ the number of cycles of $\sigma$ of length $k,$ by  $\#\,\sigma$ the total number of cycles of $\sigma$  and by $\textrm{tr}(\sigma):=\#_1\,\sigma$.
\paragraph*{}
The cycle structure of a permutation chosen uniformly among the symmetric group  $\s$ is  well understood (see e.g. \citep*{9783037190005} for detailed results). In particular, the following classical result holds:
\begin{theorem}\citep*[Theorem  3.1]{arratia2000} \label{thm-B}
If $\sigma_n$ follows the uniform distribution on $\s$ then for any $k\geq 1$,
\begin{align}\label{maineq}
(\#_1\,\sigma_n, \dots, \#_k\,\sigma_n) 
\xrightarrow[n\to \infty]{d}
 \eta_k:=(\xi_1,\xi_2,\dots,\xi_k),
\end{align}
where $\xrightarrow[n\to \infty]{d}$ denotes the convergence in distribution, 
$\xi_1,\xi_2,\dots \xi_k$ are independent and the distribution of $\xi_d$ is Poisson of parameter $\frac 1d$.
\end{theorem}
In this work, we question the universality class of this convergence. We show that a product of conjugation invariant permutations that do not have too many fixed points and cycles of size 2 lies within this class. More precisely, we have the following.
\begin{theorem} \label{cor}
Let $m \geq 2$. For $1\leq \ell \leq m$, let 
$(\sigma^{\ell}_n)_{n\geq 1}$  be a sequence of random permutations  such that for any $n\geq 1$, $\sigma^{\ell}_n \in \s.$ For any $k \ge 1,$ let $t^n_k:=\#_k(\prod_{\ell=1}^m \sigma^{\ell}_n).$
Assume that
\begin{itemize}
    \item[-] ($H_1$) For any $n\geq 1$, $(\sigma^1_n, \dots, \sigma^{\ell}_n)$ are independent. 
    \item[-] For any $n\geq 1$ and $1 \le \ell \le m,$ for any $\sigma \in \s,$
     \begin{align}
       \tag{$H_2$}  \sigma^{-1}\sigma_n^\ell\sigma\overset{d}{=}\sigma_n^\ell,
     \end{align}
 except maybe for one $\ell \in \{1,\ldots,m\}.$ 
    \item[-] There exists $1\leq i < j\leq m$ such that for any $k \geq 1$,  
    \begin{align}
    \tag{$H_3$}  \lim_{n\to\infty} \E\left(\left(\frac{\#_1 \,\sigma^{i}_n}{\sqrt n}\right)^k\right)=0 &\quad  \textrm{ and } \quad \lim_{n\to\infty} \E\left(\left(\frac{\#_1 \,\sigma^{j}_n}{\sqrt n}\right)^k\right) =0, \\
     \tag{$H_4$} \lim_{n\to\infty} \frac{\E(\#_2\,\sigma^i_n)}{n} =0 &\quad  \textrm{ and } \quad  \lim_{n\to\infty} \frac{\E(\#_2\,\sigma^j_n)}{n} =0.    \end{align}
\end{itemize}
%\color{red} affichage à revoir \color{black}
Then for any $k\geq 1$,
\begin{align*}
(t^n_1,t^n_2,\dots,t^n_k) \xrightarrow[n\to \infty]{d} \eta_k.
\end{align*}
\end{theorem} 
This convergence has also been obtained by \cite{Muk16} for a quite different class of permutations, namely the permutations that are equicontinuous in both coordinates and converging as a permuton (see Definitions in  \citep{Muk16}). Here, it is easy to check that for any $\theta \in [0,1],$ the Ewens distribution with parameter $\theta$ satisfies the convergences required in $H_3$ and $H_4.$ Our result tells that the product of (at least two) Ewens distributions behaves like a uniform permutation, as far as small cycles are concerned.

In our framework, in the case of two permutations, a weaker result can be obtained without any hypothesis on the cycles of size 2.
\begin{proposition} \label{prop}
When $m=2,$ under $H_1, H_2$ and $H_3$, we have  convergence of the  first moment i.e for  any $v\geq 1$, 
$$\lim_{n\to\infty}\E(t^n_v)= \frac{1}{v}.$$
\end{proposition}

% \begin{theorem} \label{cor}
% Let $(\sigma_n)_{n\geq 1}$ and $(\rho_n)_{n\geq 1}$ be two sequences of random permutations  such that for any $n\geq 1$, $\sigma_n,\rho_n \in \s$ and let $t^n_k:=\#_k(\sigma_n\rho_n)$.
% Assume that
% \begin{itemize}
%     \item[$H_
%     1$:] For any $n\geq 1$, $\rho_n$ and $\sigma_n$ are independent. 
%     \item[$H_2$:] For any $n \geq 1$, for any $\sigma \in \s$, 
%     \begin{align*}
%         \sigma^{-1}\sigma_n\sigma\overset{d}{=}\sigma_n.
%     \end{align*}
%     \item[$H_3$:] For any $k \geq 1$,  $$\lim_{n\to\infty}n^{\frac{-k}{2}} \E(\#_1(\sigma_n)^k+\#_1(\rho_n)^k)=0. $$
%         \item[$H_4$:]  
%         $$ \lim_{n\to\infty} \frac{\E(\#_2(\sigma_n)+\#_2(\rho_n))}{n}=0. $$
% \end{itemize}
% Then for any $k\geq 1$,
% \begin{align*}
% (t^n_1,t^n_2,\dots,t^n_k) \xrightarrow[n\to \infty]{d} \eta_k.
% \end{align*}
% \end{theorem}
Note that when one of the permutations $\sigma_n^\ell$ follows the uniform distribution, under $H_1$, the product also follows the uniform distribution and  Theorem~\ref{cor} is a direct consequence of Theorem~\ref{thm-B}. %Note also that when at lboth $\sigma_n$ and $\rho_n$ follows the uniform distribution then  $H_3$ and $H_4$ are satisfied.

%Otherwise stated,   $\forall v \geq 1$
%$$\lim_{n\to\infty}n\mathbb{P}\left(c_1\left(\sigma^{-1}_n\circ\rho_n\right)=v\right)= 1. $$
Our motivation to understand the cycle structure of random permutations is the relation, in the case of conjugation invariant permutations, to the longest common subsequence (LCS) of two permutations. For example, for $m =2,$ if $\sigma_n^{-1}\rho_n$ is conjugation invariant and 
\begin{align*}
 \frac{\#(\sigma_n^{-1}\rho_n)}{\sqrt[6]{n}} \xrightarrow[n\to \infty]{d} 0.
\end{align*} Then for any $s\in \mathbb{R}$,
$$\p\left(\frac{LCS(\sigma_n,\rho_n)-2\sqrt{n}}{{\sqrt[6]{n}}}\leq s\right)\xrightarrow[n\to \infty]{} F_2(s),$$
where $F_2$ is the cumulative distribution function of the GUE Tracy-Widom distribution.

Another motivation comes from traffic distributions,  a non-commutative probability theory introduced by \cite{2011arXiv1111.4662M} to understand the moments of 
permutation invariant random matrices. As shown in \citep{2011arXiv1111.4662M}, the limit in traffic distribution of uniform permutation matrices is trivial but Theorem~\ref{thm-B} can be seen as a second-order result in this framework. It is therefore natural to ask about limiting joint fluctuations for the  
product of several permutation matrices, which is a really non-commutative case.
%\begin{theorem}\label{THM}
%Let $(A_n)_{n\geq 1},{B_n}_{n\geq 1}$ be %two sequences of random  permutation  %matrices such that for any $n\geq 1$, %$A_n,B_n \in \s$.
%Assume that
%\begin{itemize}
%    \item[H_1 :] For any $n\geq 1$, %$\sigma_n$ and $B_n$ are independent. 
%    \item[H_2 :] For any $n \geq 1$, for any %$\sigma \in \s$, 
%    \begin{align*}
%        \sigma^{-1}A_n\sigma\overset{d}{=}A%_n.
%    \end{align*}
%    \item[H_3 :] For any $k \geq 1$,  %$${\sup_{n}\E((tr(A_n)+tr(B_n))^k)}<\infty.%$$
%\end{itemize}
%Then for any $k\geq 1$,
%$$\lim_{n\to\infty}\E(tr((A_nB_n)^k))=card(%\{1\leq d\leq k; d|k\}).$$
%\end{theorem}
To emphasize this relation, we rewrite Theorem~\ref{cor} as follows.
\begin{corollary} \label{cor2}
Under $H_1$, $H_2$, $H_3$ and $H_4$,
 for any $k\geq 1$,
$\left(\mathrm{tr}(\prod_{i=1}^m\sigma_n^\ell),\mathrm{tr}((\prod_{i=1}^m\sigma_n^\ell)^2),\dots,\mathrm{tr}((\prod_{i=1}^m\sigma_n^\ell)^k)\right)$ converges in distribution to $(\xi_1,\xi_1+ 2\xi_2,\dots,\sum_{d|k} d\xi_d)$, where $\xi_1,\xi_2,\dots$ are independent and the distribution of $\xi_d$ is Poisson of parameter $\frac 1d$.
%Consequently,
%for any positive integers $k$ and $j$,

%$$j\E(tr((\sigma_n\rho_n)^j)^k) \to   \sum_{i=0}^k (\sum_{d|j} d)^i \left\{\begin{matrix} k \\ i \end{matrix}\right\},
%$$
%where $\left\{\begin{matrix} k \\ i %\end{matrix}\right\}$ are Stirling numbers %of the second kind.
\end{corollary}
% Under stronger assumptions, \cite[Theorem 1.8]{2011arXiv1111.4662M}  states the existence of such a limit and using the same definitions as in \citep{2011arXiv1111.4662M},  $H_1$ can be replaced by  asymptotically traffic free (see Remark~\ref{rm1}).
The optimality of conditions $H_3$ and $H_4$ will be discussed at the end of the paper. %in Remark~\ref{opt}.
\paragraph*{Acknowledgements :}
The first author would like to acknowledge  a useful discussion with Camille Male  about traffic distributions. This work is partially supported  by the Labex CEMPI (ANR-11-LABX-0007-01). 
%\begin{corollary} \label{cor3}
%Let $j\geq 2$, $(e\sigma^i_n)_{n\geq 1, 1\leq i\leq j}$ be a sequences %of random permutations  such that for any $n\geq 1$, for all $ 1\leq i %\leq j$,  $\sigma^i_n\in \s$ and let $t^n_ %k:=\#_k(\prod_{i=1}^j\sigma^i_n)$.
%Assume that
%\begin{itemize}
%    \item[H'1 :] For any $n\geq 1$, $\sigma_n^1,\sigma_n^1,\dots,\sigma_n^j$  are independent. 
%    \item[H'2 :] All permutations (expect perhaps one) are invariant under conjugation
%    \item[H'3 :]  There exists $i_1\neq i_2$ such that  for any $k \geq 1$, $$\sup_{n}\E(\#_1(\sigma^{i_1}_n)^k+\#_1(\sigma^{i_1}_n)^k)<\infty.$$
%\end{itemize}
%Then for any $k\geq 1$,
%$(t^n_1,t^n_2,\dots,t^n_k)$ converges in dis\textrm{tr}ibution to 
%$(\xi_1,\xi_2,\dots,\xi_k)$,
%where $\xi_1,\xi_2,\dots$ are independent and the distribution of %$\xi_d$ is Poisson of parameter $\frac 1d$.
%\end{corollary}
\section{Proof of results}
We begin with a few preliminary remarks and simplifications.

First of all, the equivalence between Theorem~\ref{cor} and Corollary~\ref{cor2} is due to the following classical argument. For any $\sigma \in \s$, if $c_i(\sigma)$ denotes the length of the cycle of $\sigma$ containing $i,$
\begin{equation} \label{trace-cycles}
    \textrm{tr}(\sigma^k)= \sum_{i=1}^n \mathbbm{1}_{\sigma^k(i)=i}=\sum_{i=1}^n \mathbbm{1}_{c_i(\sigma)|k} = \sum_{j|k}\sum_{i=1}^n\mathbbm{1}_{c_i(\sigma)=j}=\sum_{j|k} j \, \#_j \sigma.
\end{equation}
% We know also that the moment of order $k$ of a Poisson variable of parameter $\lambda$ is 
%$$m_k = \sum_{i=0}^k \lambda^i \left\{\begin{matrix} k \\ i \end{matrix}\right\},$$
%where $\left\{\begin{matrix} k \\ i \end{matrix}\right\}$ are Stirling numbers of the second kind.

In the hypothesis $H_2,$ we assume that one of the  permutations, say  $\sigma^1_n,$ may not have a  conjugation invariant distribution. In fact, it is enough to
 prove of Theorem~\ref{cor} in the case where all permutations are conjugation invariant. Indeed, if we choose $\tau_n$ uniform and independent of  the $\sigma$-algebra generated by $(\sigma^\ell_n)_{1\leq \ell\leq  m}$, the cycle structure of $\prod_{\ell=1}^m \sigma^\ell_n$ is the same as  
\begin{align*}
    \tau_n^{-1}\left(\prod_{\ell=1}^m \sigma^\ell_n\right)\tau_n=(\tau_n^{-1}\sigma_n^{1}\tau_n) \prod_{\ell=2}^m (\tau_n^{-1} \sigma^\ell_n\tau_n)\overset{d}{=}(\tau_n^{-1}\sigma_n^{1}\tau_n) \prod_{\ell=2}^m  \sigma^\ell_n
\end{align*}
and $(\tau_n^{-1}\sigma^1_n\tau_n)$ is also conjugation invariant.

We will prove in full details the case $m=2$ and indicate briefly at the end of the paper how to extend the proof to a larger number of permutations. In the sequel, $\sigma_n^1$ and $\sigma_n^2$ will be denoted respectively by $\sigma_n$ and $\rho_n.$

%For sake of clarity,  we will prove first a weaker version of Theorem~\ref{cor} stated in the following proposition.
%\begin{proposition} \label{propositon faible}
%Let $(\sigma_n)_{n\geq 1},(\rho_n)_{n\geq 1}$ be two sequences of random permutations  such that for any $n\geq 1$, $\sigma_n,\rho_n \in \s$ and let $\hat{t}^n_k:=\#_k(\sigma_n^{-1}\rho_n)$.
%Assume $H_1,H_4$ and that
%\begin{itemize} 
 %   \item[H'2 :] For any $n \geq 1$, for any $\sigma \in \s$,   
  %  $    \sigma^{-1}\rho_n\sigma\overset{d}{=}\rho_n \quad \text{and} \quad \sigma^{-1}\sigma_n\sigma\overset{d}{=}\sigma_n.
%$
%\item[H'3 :]  For any $k$, $\E(\#_1(\sigma_n)+\#_1(\rho_n))= 0.$
%\end{itemize}
%Then for any $k\geq 1$, for any monomial
% function $P$, $\E(P(\hat{t}^n_1,\hat{t}^n_2,\dots,\hat{t}^n_k))$ converges  to a constant depending only on $P$.
%\end{proposition}

\subsection{Preliminary results}

To prove Theorem~\ref{cor}, we will use the same objects introduced in \citep[pages  12-13]{kam2} where one can get further details and examples.  To a couple of permutations and a subset of $p$ indices, we will associate a set of $2p$ graphs. For technical reasons, we prefer working with $\sigma_n^{-1}\rho_n$ rather than 
$\sigma_n\rho_n :$ for any $k \ge 1,$ we define $\tilde{t}_k^n:=\#_k(\sigma_n^{-1}\rho_n).$
Under $H_2$,  $\sigma_n\overset{d}=\sigma^{-1}_n$ and consequently under $H_1$ and $H_2$, $\forall k\geq 1$ $(t_1^n,t_2^n,\dots,t_k^n)$ and $(\tilde{t}_1^n,\tilde{t}_2^n,\dots,\tilde{t}_k^n)$  have the same distribution. 

Let us now recall the combinatorial objects we will use. 
\begin{itemize}
 \item We denote by $\mathbb{G}^n_k$ the set of oriented simple graphs  with vertices $\{1,2,\dots,n\}$ and having exactly $k$ edges. Given  $g\in\mathbb{G}^n_k$, we denote by $E_g$ the set of its edges and by $A_g:=[\mathbbm{1}_{(i,j)\in E_g}]_{1\leq i,j\leq n}$ its adjacency matrix. 
 \item A connected component of $g$ is called \textit{trivial} if it does not have any edge and a vertex $i$ of $g$ is called \textit{isolated} if $E_g$ does not contain any edge of the form $(i,j)$ or $(j,i)$ nor a loop $(i,i)$. Let $g\in\mathbb{G}^n_k $, we denote by  $\tilde{g}$  the graph obtained  from  $g$  after removing isolated vertices.
 \item We say that  two  oriented simple graphs $g_1$ and $g_2$ are \textit{isomorphic} if  one can obtain  $g_2$ by changing the labels of the vertices of $g_1$. In particular, if  $g_1,g_2\in\mathbb{G}^n_k$  then $g_1,g_2$ are isomorphic if and  only if   there exists a permutation matrix $\sigma$ such that  $A_{g_1}\sigma=\sigma A_{g_2}$.
 \item Let $\mathcal{R}$ be the equivalence relation  such that $g_1\mathcal{R} g_2$ if  $\tilde{g}_1$ and $\tilde{g}_2$ are isomorphic.
 We denote by  $\hat{\mathbb{G}}_k:=\bigslant{\cup_{n\geq1} \mathbb{G}_k^n}{\mathcal{R}}
$   the set of  equivalence classes of $\cup_{n\geq1} \mathbb{G}_k^n$  for the relation $\mathcal{R}$.
\end{itemize}

Let $n\in \mathbb{N}^*$  and $\sigma,\rho\in\s.$ Let $m \in \{1, \ldots, n\}$ be fixed.
\begin{itemize}
 \item We denote by $(i^m_1(\sigma,\rho)=m,i^m_2(\sigma,\rho),\dots,i^m_{k_m(\sigma,\rho)}(\sigma,\rho))$  the cycle  of $\sigma^{-1}\circ\rho$ containing $m,$ so that
 $k_m(\sigma,\rho):=c_m(\sigma^{-1}\circ\rho)$ is the length of this cycle. For $i \le k_m(\sigma, \rho) ,$ we define $j_l^m(\sigma,\rho):=\rho(i^m_l(\sigma,\rho))$. 
 In particular, $i_1^m(\sigma,\rho),i_2^m(\sigma,\rho),\dots,i_{k_m(\sigma,\rho)}^m(\sigma,\rho)$ are pairwise distinct and  $j_1^m(\sigma,\rho),j_2^m(\sigma,\rho),\dots,$ \linebreak $j_{k_m(\sigma,\rho)}^m(\sigma,\rho)$  are pairwise distinct.
For sake of simplicity, when it is clear, we will use the notations $k_m$, $i^m_l$ and $j^m_l$ instead of $k_m(\sigma,\rho)$, $i^m_l(\sigma,\rho)$ and $j^m_l(\sigma,\rho)$.
\item We denote by   $\mathcal{G}_1^m (\sigma,\rho) \in \mathbb{G}_{k_m}^n$ and $\mathcal{G}_2^m(\sigma,\rho)\in \mathbb{G}_{k_m}^n$  the  graphs with vertices  $\{1, \ldots, n\}$ such that 
\[ E_{\mathcal{G}_1^m (\sigma,\rho)}=\{(i^m_1,j^m_{k_m})\} \bigcup \left(\bigcup_{l=1}^{k_m-1}{\{(i^m_{l+1},j^m_{l})\}}\right)  \quad \textrm{ and } \quad
E_{\mathcal{G}_2^m (\sigma,\rho)}=\bigcup_{l=1}^{k_m}{\{(i^m_l,j^m_{l})\}}\]
and by $g_\sigma$ the graph such that $A_{g_\sigma}=\sigma$. 
By construction, for any positive  integer $m\leq n$, $\mathcal{G}_1^m(\sigma,\rho)$ (resp. $\mathcal{G}_2^m(\sigma,\rho)$) is a sub-graph of $g_\sigma$ (resp. $g_\rho$).
Moreover, we want to emphasize that $\mathcal{G}_1^m(\sigma,\rho)$ and $\mathcal{G}_2^m(\sigma,\rho)$ have the same set of non-isolated vertices.

For $i\in\{1,2\}$, let 
$\hat{\mathcal{G}}^m_i(\sigma,\rho)$ be the  equivalence class of $\mathcal{G}^m_i(\sigma,\rho)$. 
\item Let $I=(s_1,s_2,\dots,s_l)$ a set of distinct indices of $\{1, \ldots,n\}.$ We denote by $$\mathcal{G}^I(\sigma,\rho)=(\mathcal{G}^{s_1}_1(\sigma,\rho),\mathcal{G}^{s_1}_2(\sigma,\rho),\mathcal{G}^{s_2}_1(\sigma,\rho),\dots,\mathcal{G}^{s_l}_1(\sigma,\rho),\mathcal{G}^{s_l}_2(\sigma,\rho))$$
and 
$$\hat{\mathcal{G}}^I(\sigma,\rho)=(\hat{\mathcal{G}}^{s_1}_1(\sigma,\rho),\hat{\mathcal{G}}^{s_1}_2(\sigma,\rho),\hat{\mathcal{G}}^{s_2}_1(\sigma,\rho),\dots,\hat{\mathcal{G}}^{s_l}_1(\sigma,\rho),\hat{\mathcal{G}}^{s_l}_2(\sigma,\rho)).$$
\item For $i \in \{1,2\}$, let $\mathcal{G}_i^{\{1,2,\dots,k\}}(\sigma,\rho)$  be the graph such that $E_{\mathcal{G}_i^{\{1,2,\dots,k\}}(\sigma,\rho)}=\cup_{l=1}^k E_{\mathcal{G}_i^{\ell}(\sigma,\rho)}$ 
and
$\hat{\mathcal{G}}_i^{\{1,2,\dots,k\}}(\sigma,\rho)$ be the equivalence class of $\mathcal{G}_i^{\{1,2,\dots,k\}}(\sigma,\rho)$.

\end{itemize}

Using the conjugation invariance and the relation \eqref{trace-cycles},
%$$  \#_k(\sigma)= \sum_{i=1}^n \frac{\mathbbm{1}_{c_i(\sigma)=k}}{k},$$
Theorem~\ref{cor} is equivalent to the following: under the same hypotheses, for any $v_1,v_2,v_3,\dots,v_k \geq 1$,
\begin{align}\label{st}\tag{*}
\lim_{n\to \infty}  \sum_{ \substack{  \hat{g}_i,\hat{g}'_i \in \hat{\mathbb{G}}_{v_i}, \, 1\leq i\leq k }}n^k
\p\left(\hat{\mathcal{G}}^{\{1,2,\dots,k\}}(\sigma_n,\rho_n)=(\hat{g}_1,\hat{g}'_1,\hat{g}_2,\dots \hat{g}'_k )\right)=C_{v_1,v_2,\dots,v_k},
\end{align}
where $C_{v_1,v_2,\dots,v_k}$ is a constant independent of the laws of the permutations. Note that, for any $v_i \geq 1,$ $\hat{\mathbb{G}}_{v_i}$ and therefore 
the number of terms of the sum is finite.
\\ For example, if we take $P(x)=x^2$, we have
\begin{align*}
    \E\left(P\left(\hat{t}^n_1\right)\right)= \E\left(\left(\sum_{i=1}^n \mathbbm{1}_{c_i\left(\sigma^{-1}\circ\rho\right)=1}\right)^2\right)
    &= \sum_{i=1}^n \E\left(\mathbbm{1}_{c_i\left(\sigma^{-1}\circ\rho\right)=1}\right) +    \sum_{i\neq j}^n \E\left(\mathbbm{1}_{c_i\left(\sigma^{-1}\circ\rho\right)=1}\mathbbm{1}_{c_j\left(\sigma^{-1}\circ\rho\right)=1}\right)
    \\&= n\E\left(\mathbbm{1}_{c_1\left(\sigma^{-1}\circ\rho\right)=1}\right)+ (n^2-n) \E\left(\mathbbm{1}_{c_1\left(\sigma^{-1}\circ\rho\right)=1}\mathbbm{1}_{c_2\left(\sigma^{-1}\circ\rho\right)=1}\right)
    \\ &\xrightarrow[n\to\infty]{} C_1+C_{1,1}=1+1=2
\end{align*} 
Similarly, if we take $P(x,y)=xy$, we obtain $\E(P(\hat{t}^n_1,\hat{t}^n_2))\xrightarrow[n\to\infty]{d} C_{1,2}=C_{2,1}=1.$

Before getting into the proof of \eqref{st}, let us gather some useful combinatorial and then probabilistic results.
\begin{lemma}
\cite[Lemma 15]{kam2}
\label{lem20}If  ${m_1}\in \{i^{m_2}_l, 1\leq l\leq k_{m_2}\}$, then $\mathcal{G}^{m_1}_1(\sigma,\rho)=\mathcal{G}^{m_2}_1(\sigma,\rho)$ and $\mathcal{G}^{m_1}_2(\sigma,\rho)=\mathcal{G}^{m_2}_2(\sigma,\rho)$.
\end{lemma}

\begin{lemma}\label{sym}
For any $m \leq n $, for any permutation $\sigma,\rho \in \s$,
\begin{align*}
k_m(\rho,\sigma)&=k_m(\sigma,\rho),
\\
j^m_\ell(\rho,\sigma)&=j^{m}_{k_m(\sigma,\rho)-\ell+1}(\sigma,\rho), \ \forall  1\leq \ell \leq k_m(\sigma,\rho),
\\
i^m_\ell(\rho,\sigma)&=i^{m}_{k_m(\sigma,\rho)-\ell+2}(\sigma,\rho), \ \forall  2\leq \ell \leq k_m(\sigma,\rho),
\\
i^m_1(\rho,\sigma)&=i^{m}_{1}(\sigma,\rho)=m,
\\ A_{\mathcal{G}_1^{m}(\sigma,\rho)}&=A_{\mathcal{G}_2^{\rho(m)}(\rho^{-1},\sigma^{-1})}^T.\end{align*}

\end{lemma}

\pagebreak

\begin{lemma}\label{l14}
If all non trivial connected components of $\mathcal{G}^{m_1}_1(\sigma,\rho)$ and $\mathcal{G}^{m_1}_2(\sigma,\rho)$ have $2$ vertices then 
both $\mathcal{G}^{m_1}_1(\sigma,\rho)$ and $\mathcal{G}^{m_1}_2(\sigma,\rho)$ have no 2-cycles .
\end{lemma}
\begin{proof}
Using the symmetries of the problem (Lemmas \ref{lem20} and \ref{sym}), it suffices to prove  that  if all non trivial connected components of $\mathcal{G}^{1}_1(\sigma,\rho)$ and $\mathcal{G}^{1}_2(\sigma,\rho)$ have $2$ vertices then it is impossible to have at the same time   $(1,2)\in \mathcal{G}^{1}_2(\sigma,\rho) $ and $(2,1)  \in \mathcal{G}^{1}_2(\sigma,\rho)$.  To simplify notations, let  $k_1:=k_1(\sigma,\rho)=c_1(\sigma^{-1}\circ\rho)$, $i^1_o:=i^1_o(\sigma,\rho)$ and $j^1_o:=j^1_o(\sigma,\rho)$.

Let $A=\{ \eta>1 ;  j^1_{\eta}\in\{i^1_1,i_2^1,\dots,i_{\eta-1}^1\} \text{ or } i^1_{\eta}\in\{j^1_1,j_2^1,\dots,j_{\eta-1}^1\}  \} $.
 Suppose that $(1,2)\in \mathcal{G}^{1}_2(\sigma,\rho) $ and $(2,1)  \in \mathcal{G}^{1}_2(\sigma,\rho)$ then $k_1\geq 2$ and there exists a unique   $1<l\leq k_1$ such that $i^1_l=2$ and $j^1_l=1$ so that  $A$ is non-empty. 
 Let $\ell' := inf(A)\geq 2 $. Assume that $\ell'>2$. If $j^1_{\ell'}\in\{i^1_1,i_2^1,\dots,i_{\ell'-1}^1\}$, then there exists $\ell''<\ell'$ such that $j^1_{\ell'}=i^1_{\ell''}$ and since the component  of  $\mathcal{G}^{1}_2(\sigma,\rho)$ containing $i^1_{\ell'}$ has two vertices and by definition $(i^1_{\ell'},j^1_{\ell'})$ and $(i^1_{\ell''},j^1_{\ell''})$ are two edges of 
 $\mathcal{G}^{1}_2(\sigma,\rho)$, then $j^1_{\ell''}=i^1_{\ell'}$. %Note that necessarily $\ell'<k_1(\sigma,\rho)$, otherwise $j_{k_1}=1$ and $\mathcal{G}^{1}_2(\sigma,\rho)$ contains a loop. 
 Since  $(i^1_{\ell'},j^1_{\ell'-1})=(j^1_{\ell''},j^1_{\ell'-1})$ and
 $(i^1_{\ell''+1},j^1_{\ell''})$ are edges of $\mathcal{G}^{1}_1(\sigma,\rho)$ and  since $\mathcal{G}^{1}_1(\sigma,\rho)$ has only connected components of size $2$, we have necessarily $i^1_{\ell''+1}=j^1_{\ell'-1}$. One can check easily  that $\ell''<\ell'-2$ otherwise either $\mathcal{G}^{1}_1(\sigma,\rho)$ or $\mathcal{G}^{1}_2(\sigma,\rho)$ has a loop.
 Indeed, if $\ell''=\ell'-2$, then $(i^1_{\ell''+1},j^1_{\ell''+1})=(j^1_{\ell'-1},j^1_{\ell''+1} )=(j^1_{\ell'-1},j^1_{\ell'-1} )$ is an edge of  $\mathcal{G}^{1}_2(\sigma,\rho)$ and if $\ell''=\ell'-1$, then $(i^1_{\ell''+1},j^1_{\ell''})=(j^1_{\ell'-1},j^1_{\ell''} )=(j^1_{\ell'-1},j^1_{\ell'-1} )$ 
 is an edge of  $\mathcal{G}^{1}_1(\sigma,\rho)$.
 This implies that $\ell'-1 \in A$, which is absurd. $i^1_{\ell'}\in\{j^1_1,j_2^1,\dots,j_{\ell'-1}^1\}$ can be treated using the same techniques and one can extend easily to  $\ell'=2$.
 \end{proof}

We now introduce the following notation : given  $g\in\mathbb{G}^n_k$, we denote by 
$$\mathfrak{S}_{n,g}:=\{\sigma\in \s; \forall (i,j)\in E_g, \sigma(i)=j \}.$$
In other words, $\mathfrak{S}_{n,g}$ is the set of permutations $\sigma$ such that $g$ is a sub-graph of $g_\sigma$.
It is not difficult to prove the two following lemmas.
\begin{lemma}\label{lem:prob}
Let $g_1,g'_1,g_2,\dots,g'_k \in \cup_{\ell} {\mathbb{G}}^n_\ell$ and let $g,g'$ be such that 
$E_g=\cup_{\ell=1}^k E_{g_i}$ and $E_{g'}=\cup_{\ell=1}^k E_{g'_i}$. 
Assume that there exists $\rho,\sigma$ such that
$$\mathcal{G}^{\{1,2,\dots,k\}}(\sigma,\rho)=(g_1,g'_1,g_2,\dots,g'_k ).$$
Then for any random permutation $\rho_n,\sigma_n$, 
\begin{align*}
   \p\left(\bigcap_{i=1}^k\{\sigma_n \in \mathfrak{S}_{n,g_i},\rho_n \in \mathfrak{S}_{n,g'_i} \}\right)
&=\p\left(\mathcal{G}^{\{1,2,\dots,k\}}(\sigma_n,\rho_n)=(g_1,g'_1,g_2,\dots,g'_k )\right)\\&= \p\left(\mathcal{G}_1^{\{1,2,\dots,k\}}(\sigma_n,\rho_n)=g,\mathcal{G}_2^{\{1,2,\dots,k\}}(\sigma_n,\rho_n)=g'\right).
\end{align*}
\end{lemma}
\begin{proof} We will only prove the first equality. The second one  can be obtained using the same argument. 

Let $\sigma',\rho'$ be two permutations.  We have seen that $\mathcal{G}^m_2(\sigma',\rho')$ is a subset of $g_{\rho'},$ so that
$$\mathcal{G}^m_2(\sigma',\rho')=g'_m \Rightarrow \rho' \in \mathfrak{S}_{n,g'_m}, $$
and that  $\mathcal{G}^m_1(\sigma',\rho')$ is a subset of $g_{\sigma'},$ so that 
$$\mathcal{G}^m_1(\sigma',\rho')=g_m \Rightarrow \sigma' \in \mathfrak{S}_{n,g_m}. $$
%where  $a\%b$ is the remainder obtained by the euclidean division of $a$ by $b$.
Consequently, 
$$\p\left(\mathcal{G}^{\{1,2,\dots,k\}}(\sigma_n,\rho_n)=(g_1,g'_1,g_2,\dots,g'_k )\right)\leq
\p\left(\bigcap_{i=1}^k\{\sigma_n \in \mathfrak{S}_{n,g_i},\rho_n \in \mathfrak{S}_{n,g'_i} \}\right).
$$
Now suppose that there exists $\rho',\sigma'$ such that
$$\mathcal{G}^{\{1,2,\dots,k\}}(\sigma',\rho')=(g_1,g'_1,g_2,\dots,g'_k ).$$
Let $\sigma,\rho$ such that  $\sigma \in \cap_{i=1}^k \mathfrak{S}_{n,g_i}$ and  $\rho \in \cap_{i=1}^k \mathfrak{S}_{n,g'_i}$. 
 By definition and  by iteration on $\ell$, one can check  that 
for any $\ell' \leq k$,  $i^\ell_{\ell'}(\sigma',\rho')=i^\ell_{\ell'}(\sigma,\rho)$ and $j^\ell_{\ell'}(\sigma',\rho')=j^\ell_{\ell'}(\sigma,\rho)$. Consequently, 
$$\mathcal{G}^{\{1,2,\dots,k\}}(\sigma,\rho)=(g_1,g'_1,g_2,\dots,g'_k ).$$
Finally we obtain 
$$\p\left(\mathcal{G}^{\{1,2,\dots,k\}}(\sigma_n,\rho_n)=(g_1,g'_1,g_2,\dots,g'_k )\right)\geq
\p\left(\bigcap_{i=1}^k\{\sigma_n \in \mathfrak{S}_{n,g_i},\rho_n \in \mathfrak{S}_{n,g'_i} \}\right).
$$
\end{proof}

\begin{lemma}\cite[Lemma 16]{kam2}\label{lemma15}
Let $g_1, g_2 \in \mathbb{G}^n_k$. Assume that there exists  $\rho \in \s$ 
such that $A_{g_2}\rho=\rho A_{g_1}$. If $\rho$ has  a fixed point on  any non-trivial connected component  of $g_1$, then $\mathfrak{S}_{n,g_1}\cap\mathfrak{S}_{n,g_2}=\emptyset $ or $A_{g_1}=A_{g_2}$.
\end{lemma}
\begin{lemma} \label{b111}
For any graph $g\in \mathbb{G}^n_k$ having $f$ loops,  $p$ non-trivial connected components and $v$ non-isolated vertices, for any random permutation $\sigma_n$ with conjugation invariant distribution on $\s$,
\begin{align*}
    \p(\sigma_n\in\mathfrak{S}_{n,g}) \leq \frac{\p(\sigma_n(1)=1,\dots, \sigma_n(f)=f)}{{\binom{n-p}{v-p}}(v-p)!}\leq \frac{1}{{\binom{n-p}{v-p}}(v-p)!}.
\end{align*}
\end{lemma}
 
\begin{proof}
It is an adaptation of the proof of   \cite[Corollary 17]{kam2}. By conjugation invariance, one can suppose without loss of generality that the loops of $g$  are $(1,1),(2,2),\dots (f,f)$ and the set of non isolated vertices of $g$ are $\{1,2,\dots,v\}$. 
\\
If there exist $i,j,l$, with $j\neq l $  such that $\{(i,j)\cup(i,l)\} \subset E_g$ or $\{(j,i)\cup(l,i)\} \subset E_g$  then $\mathfrak{S}_{n,g} =\emptyset$. Therefore, if $\mathfrak{S}_{n,g} \neq \emptyset$, then  non-trivial connected components of $g$ having $w$ vertices  are either cycles of length $w$ or isomorphic to $\overline{g}_w$, where $
A_{\overline{g}_w}=[\mathbbm{1}_{j=i+1}]_{1\leq i,j\leq w}.$\\
Let $g\in \mathbb{G}^n_k$ such that $\mathfrak{S}_{n,g} \neq \emptyset$.
Fix  $p$ vertices $x_1=1,x_2=2,\dots,x_f=f,x_{f+1},\dots,x_p$  each belonging  to a different  non-trivial connected components of $g$.  Let $x_{p+1}<x_{p+2}<\dots<x_v$ be such that $\{x_{p+1}, \dots, x_v\} = \{1,2,\dots,v\}\setminus\{x_1,\dots x_p\}$ be the other non-isolated vertices.
Let $$F= \{(y_i)_{p+1\leq i \leq v}; y_i \in \{1,2,\dots,n\}\setminus\{x_1,\dots x_p\} \text{ pairwise distinct}\}.$$ 
Given $y=(y_i)_{p+1\leq i \leq v}\in F$, we denote by $g_y \in \mathbb{G}^n_k$ the graph isomorphic to $g$ obtained by fixing the labels of $ x_1,x_2,\dots,x_p$ and by changing the labels of $x_i$ by $y_i$ for $p+1\leq i\leq v$. Since non trivial connected  components of $g$ of length $w$ are either cycles or isomorphic to $\bar{g}_w$, if $y \neq y' \in F$, then $g_y\neq g_{y'}$ and by Lemma~\ref{lemma15}, $\mathfrak{S}_{n,g_{y}}\cap\mathfrak{S}_{n,g_{y'}}=\emptyset$. Since $\sigma_n$ is conjugation invariant, we have $ \p(\sigma_n\in\mathfrak{S}_{n,g_{y}})= \p(\sigma_n\in\mathfrak{S}_{n,g_{y'}})= \p(\sigma_n\in\mathfrak{S}_{n,g})$. Remark also that for any  $y\in F$ and any $i\leq f$, $(i,i)$ is a loop of $g_y$. Thus, $\mathfrak{S}_{n,g_{y}} \subset \{\sigma \in \s; \forall i \leq f, \sigma_n(i)=i\}$ and thus
\begin{align*}
     \p(\sigma_n\in\mathfrak{S}_{n,g})=\frac{\sum_{y\in F}\p(\sigma_n\in\mathfrak{S}_{n,g_{y}})}{\textrm{card}{(F)}}=\frac{\p(\sigma_n\in\cup_{y\in F}\mathfrak{S}_{n,g_{y}})}{\textrm{card}(F)}&\leq  \frac{\p(\sigma_n(1)=1,\dots, \sigma_n(f)=f)}{{\binom{n-p}{v-p}}(v-p)!}\\
     &\leq \frac{1}{{\binom{n-p}{v-p}}(v-p)!}.
\end{align*}

\end{proof}

 \begin{lemma}\label{lem:negligible}
Let $\sigma_n$ be a random permutation with conjugation invariant distribution on $\s$ such that, for any $k \ge 1,$  $\lim_{n\to\infty} \E\left(\left(\frac{\#_1 \,\sigma_n}{\sqrt n}\right)^k\right)=0.$ Then, for any $f\geq1,$
 $$ \p(\sigma^1_n(1)=1,\dots, \sigma^1_n(f)=f)=o(n^{-\frac{f}{2}}).$$
 \end{lemma}
\begin{lemma} \label{lem172} For any $p\geq 1$,
let  $g$  be a graph with $p$ non trivial components each having $2$ vertices. Assume that at least one of these components is a cycle. Then  for any random permutation  $\sigma_n$  with conjugation invariant distribution on $\s$,
\begin{align*}
     \p(\sigma_n\in\mathfrak{S}_{n,g}) \leq \frac{\p(c_1(\sigma_n)=2)}{{\binom{n-p}{p}}p!}.
\end{align*}
\end{lemma}
\begin{proof}
Remark that  by conjugation invariance, one can suppose without loss of generality that the set of non isolated vertices of $g$ are $\{1,2,\dots,2p\}$ and that  $(1,2),(2,1)\in E_g$. Using the same definitions as the previous proof with $f=0$ and $v=2p$ and by choosing $x_1=1$, we have   
$\mathfrak{S}_{n,g_{y}} \subset \{\sigma \in \s; c_1(\sigma)=2\}.$ Thus, 
\begin{align*}
     \p(\sigma_n\in\mathfrak{S}_{n,g})=\frac{\sum_{y\in F}\p(\sigma_n\in\mathfrak{S}_{n,g_{y}})}{\textrm{card}{(F)}}=\frac{\p(\sigma_n\in\cup_{y\in F}\mathfrak{S}_{n,g_{y}})}{\textrm{card}(F)}&\leq \frac{\p(c_1(\sigma_n)=2)}{\textrm{card}(F)}
     = \frac{\p(c_1(\sigma_n)=2)}{{\binom{n-p}{p}}p!}.
\end{align*}
 \end{proof}

By the previous combinatorial lemmas, we get that the main contribution will come from the following subset of graphs.
Let $\mathcal{T}^n_k\subset \mathbb{G}^n_k$ be  the set of graphs $g$  having exactly $k$ non trivial component each having one edge and two vertices.    \\ For example, $\mathcal{T}^3_1= \left\{\begin {tikzpicture}[-latex ,auto ,node distance =1 cm and 1cm ,on grid ,
semithick ,
state/.style ={ circle ,top color =white , bottom color = processblue!20 ,
draw,processblue , text=blue , minimum width =0.1 cm}]
\node[state] (C) {$1$};
\node[state] (D) [right =of C] {$2$};
\path (C) edge [bend left =25]  (D); 
\end{tikzpicture}, 
\begin {tikzpicture}[-latex ,auto ,node distance =1 cm and 1cm ,on grid ,
semithick ,
state/.style ={ circle ,top color =white , bottom color = processblue!20 ,
draw,processblue , text=blue , minimum width =0.1 cm}]
\node[state] (C) {$2$};
\node[state] (D) [right =of C] {$1$};
\path (C) edge [bend left =25]  (D); 
\end{tikzpicture}, 
\begin {tikzpicture}[-latex ,auto ,node distance =1 cm and 1cm ,on grid ,
semithick ,
state/.style ={ circle ,top color =white , bottom color = processblue!20 ,
draw,processblue , text=blue , minimum width =0.1 cm}]
\node[state] (C) {$1$};
\node[state] (D) [right =of C] {$3$};
\path (C) edge [bend left =25]  (D); 
\end{tikzpicture}, 
\begin {tikzpicture}[-latex ,auto ,node distance =1 cm and 1cm ,on grid ,
semithick ,
state/.style ={ circle ,top color =white , bottom color = processblue!20 ,
draw,processblue , text=blue , minimum width =0.1 cm}]
\node[state] (C) {$3$};
\node[state] (D) [right =of C] {$1$};
\path (C) edge [bend left =25]  (D); 
\end{tikzpicture}, 
\begin {tikzpicture}[-latex ,auto ,node distance =1 cm and 1cm ,on grid ,
semithick ,
state/.style ={ circle ,top color =white , bottom color = processblue!20 ,
draw,processblue , text=blue , minimum width =0.1 cm}]
\node[state] (C) {$2$};
\node[state] (D) [right =of C] {$3$};
\path (C) edge [bend left =25]  (D); 
\end{tikzpicture}, 
\begin {tikzpicture}[-latex ,auto ,node distance =1 cm and 1cm ,on grid ,
semithick ,
state/.style ={ circle ,top color =white , bottom color = processblue!20 ,
draw,processblue , text=blue , minimum width =0.1 cm}]
\node[state] (C) {$3$};
\node[state] (D) [right =of C] {$2$};
\path (C) edge [bend left =25]  (D); 
\end{tikzpicture} 
\right\}.$ Let $\widehat{\mathcal{T}}_k $ be the equivalence class of the graphs of $\cup_{n}\mathcal{T}^n_k$. 
\\ 

%\SK{Pb: la loi n'entre pas dans cette classe attention}

%\SK{à définir proprement à partir de la loi uniform  $C_{v_1,v_2,\dots,v_n}$ }

Their contribution is as follows.
\begin{lemma} \label{lem17} For any $p\geq 1$, $n\geq2p$  and  any  graph $g \in \mathcal{T}^n_p,$  for any random permutation  $\sigma_n$  with conjugation invariant distribution    on $\s$,
\begin{align*}
    \frac{1}{{\binom{n-p}{p}}p!} \left(1-\frac{p^2-p}{n-1}-p\p(\sigma_n(1)=1)\right) \leq \p(\sigma_n\in\mathfrak{S}_{n,g}) \leq \frac{1}{{\binom{n-p}{p}}p!}.
\end{align*}
\end{lemma}
\begin{proof}
The  upper bound  is due to Lemma~\ref{b111} with $v=2p$. Using the  conjugation invariance, one can suppose without loss of generality that $E_g=\{(1,i_1),(2,i_2),\dots,(p,i_p)\}$ where $i_j>p$  are all distinct. 
Let $$\mathfrak{S}^p_{n} = \{\sigma \in \s, \forall i\leq p, \sigma(i)>p \}.$$

Remark that $\p(\sigma_n\in\mathfrak{S}_{n,g}|\sigma_n\in \s \setminus{\mathfrak{S}^p_{n})}=0$.
If $\p(\sigma_n\in \mathfrak{S}^p_{n})=0$, then necessarily by conjugation invariance, $1-\frac{p^2-p}{n-1}-p\p(\sigma_n(1)=1)\leq 0$.

Suppose now that $\p(\sigma_n\in \mathfrak{S}^p_{n})\neq0$.
We obtain
$\p(\sigma_n\in\mathfrak{S}_{n,g})=\p(\sigma_n\in\mathfrak{S}_{n,g}|\sigma_n\in \mathfrak{S}^p_{n})\p(\sigma_n\in \mathfrak{S}^p_{n}). $
Using again  the  conjugation invariance, we obtain 
$$\p(\sigma_n\in\mathfrak{S}_{n,g}|\sigma_n\in \mathfrak{S}^p_{n})=\frac{1}{{\binom{n-p}{p}}p!}$$
and
\begin{align*}
\p(\sigma_n\in \mathfrak{S}^p_{n})&=1-\p(\sigma_n\in \s \setminus \mathfrak{S}^p_{n})  
\\&\geq 1-\sum_{i=1}^p \p(\sigma_n(i)\leq p)
\\&=1-p\left( \p(\sigma_n(1)=1)+ \frac{(1-\p(\sigma_n(1)=1))(p-1)}{n} \right)
\\&\geq  1-\frac{p^2-p}{n-1}-p\p(\sigma_n(1)=1).
\end{align*}
% This concludes the proof.
\end{proof}

%\begin{lemma} \label{ll15} Assume that there exists $\rho,\sigma$ such that
%$$\mathcal{G}^{\{1,2,\dots,k\}}(\sigma,\rho)%=(g_1,g'_1,g_2,\dots,g'_k ).$$
%Then for any positive integers $k\leq n$, for any random permutation  $\rho_n,\sigma_n\in \s$,  
%$$\p(\mathcal{G}^{\{1,2,\dots,k\}}(\sigma_n,\rho_n)=(g_1,g'_1,g_2,\dots,g'_k ))=
%\p(\cap_{i=1}^k\{\sigma_n \in \mathfrak{S}_{n,g_i},\rho_n \in \mathfrak{S}_{n,g'_i} \}).
%$$
%\end{lemma}
\subsection{Proof of Proposition~\ref{prop}}
\begin{proof} We will adapt the proof of \citep[Lemma 14]{kam2}. Let $v_1 \ge 1$ be fixed.
%Note that $\hat{\mathbb{G}}_{v_1}$ is finite. Therefore, it is sufficient  to prove that for any $\hat{g}_1,\hat{g}_2 \in \hat{\mathbb{G}}_{v_1}$ having the same number of non-isolated vertices, there exists a  constant $C_{\hat{g}_1,\hat{g}_2}$  such that for any integer $n$ and under the conditions of Proposition~{\ref{prop}}
In the case $k=1$, since $C_1=1$, \eqref{st}  holds if we have: 
\begin{align*}
 \forall \hat{g}_1,\hat{g}_2  \in \hat{\mathbb{G}}_{v_1}, \p((\hat{\mathcal{G}}^1_1(\sigma_n,\rho_n),\hat{\mathcal{G}}^1_2(\sigma_n,\rho_n))=(\hat{g}_1,\hat{g}_2)))=\frac{C_{\hat{g}_1,\hat{g}_2}}{n}+ o\left(\frac{1}{n}\right)
\quad 
\text{ and } \quad \sum_{\hat{g}_1,\hat{g}_2\in \hat{\mathbb{G}}_{v_1} } C_{\hat{g}_1,\hat{g}_2}= C_1 =1.\end{align*}
Let $\hat{g}_1,\hat{g}_2 \in \hat{\mathbb{G}}_{v_1}$ be two unlabeled graphs having respectively $p_1$ and $p_2$ connected components and   $v\leq 2{v_1}$  vertices.
We denote by 
\[ p_n( \hat{g}_1,\hat{g}_2 ) :=  \p((\hat{\mathcal{G}}^1_1(\sigma_n,\rho_n),\hat{\mathcal{G}}^1_2(\sigma_n,\rho_n))=(\hat{g}_1,\hat{g}_2)).\]

Let  $B^n_{\hat{g}_1,\hat{g}_2}$ be the set of couples  $({g}_1,{g}_2) \in (\mathbb{G}^n_{v_1})^2$ having the same non-isolated vertices such that $1$ is a  non-isolated vertex of both graphs and, for $i \in \{1,2\}$, the 
equivalence class  of $g_i$ is  $\hat{g}_i$ and there exists $\sigma,\rho$ such that ${\mathcal{G}}^1_1(\sigma,\rho)=g_1$ and ${\mathcal{G}}^1_2(\sigma,\rho)=g_2.$ 
By Lemma \ref{lem:prob} and $H_1,$ we have

\begin{align}
       p_n( \hat{g}_1,\hat{g}_2 ) &= \sum_{(g_1,g_2)\in B^n_{\hat{g}_1,\hat{g}_2}} \p((\mathcal{G}^1_1(\sigma_n,\rho_n),\mathcal{G}^1_2(\sigma_n,\rho_n))=(g_1,g_2))  \nonumber
        \\
       &  = \sum_{(g_1,g_2)\in B^n_{\hat{g}_1,\hat{g}_2}}
        \p(\sigma_n\in \mathfrak{S}_{n,g_1},
        \rho_n\in \mathfrak{S}_{n,g_2})
        =\sum_{(g_1,g_2)\in B^n_{\hat{g}_1,\hat{g}_2}}
        \p(\sigma_n\in \mathfrak{S}_{n,g_1})
        \p(\rho_n\in \mathfrak{S}_{n,g_2}) \label{eq:png}
\end{align}

%\begin{itemize}
%\item[-]
%Suppose that $\hat{g}_1$ contains a loop i.e.  an edge of type $(i,i)$. In this case, 
%\begin{align*}
 %         \p((\hat{\mathcal{G}}^1_1(\sigma_n,\rho_n),\hat{\mathcal{G}}^1_2(\sigma_n,\rho_n))=(\hat{g}_1,\hat{g}_2)))&\leq \p( \exists 1\leq i \leq n ; \sigma_n(i)=i)
%          \\& = \p\left( \sum_{i=1}^n \mathbbm{1}_{\sigma_n(i)=i} \geq 1\right)
  %        \\& \leq  \E(\#_1(\sigma_n))=0.
%\end{align*} 

%Suppose that $\hat{g}_2$ contains a loop. In this case, 
%\begin{align*}
%          \p((\hat{\mathcal{G}}^1_1(\sigma_n,\rho_n),\hat{\mathcal{G}}^1_2(\sigma_n,\rho_n))=(\hat{g}_1,\hat{g}_2)))&\leq \p( \exists 1\leq i \leq n ; \rho_n(i)=i)
          %\\& = \p\left( \sum_{i=1}^n \mathbbm{1}_{\rho_n(i)=i} \geq 1\right)
        %  \\& \leq  \E(\#_1(\rho_n))=0.
%\end{align*} 
%    \item [-]

Starting from \eqref{eq:png}, we now distinguish different cases, depending on the structure of $\hat{g}_1$ and $\hat{g}_2.$
\begin{itemize}
    \item 
Case 1: $\hat{g}_1$ and $\hat{g}_2$ have respectively $f_1$ and $f_2$ loops i.e edges of type $(i,i)$ with $f_1+f_2 >0$.
Then $2p_1 -f_1 \leq {v}$ and $2p_2-f_2\leq v$. Consequently, by Lemmas \ref{b111} and \ref{lem:negligible},
\begin{align*}
       p_n( \hat{g}_1,\hat{g}_2 ) & = \, o\left(n^\frac{-f_1-f_2}{2}\right) \sum_{(g_1,g_2)\in B^n_{\hat{g}_1,\hat{g}_2}}  \frac{1}{{\binom{n-p_1}{v-p_1}}(v-p_1)!}
        \frac{1}{{\binom{n-p_2}{v-p_2}}(v-p_2)!}
        \\& =
        \frac{\textrm{card}(B^n_{\hat{g}_1,\hat{g}_2})}{{\binom{n-p_1}{v-p_1}}(v-p_1)!{\binom{n-p_2}{v-p_2}}(v-p_2)!}o\left(n^\frac{-f_1-f_2}{2}\right)
        \\ & \leq  
        \frac{{\binom{n-1}{v-1}} {v!}^2  o\left(n^\frac{-f_1-f_2}{2}\right)}{{\binom{n-p_1}{v-p_1}}(v-p_1)!{\binom{n-p_2}{v-p_2}}(v-p_2)!}
        = n^{v-1-(v-p_1+v-p_2)}o\left(n^\frac{-f_1-f_2}{2}\right) =o(n^{-1}) .
\end{align*}
\item  Case 2: $\hat{g}_1$ and $\hat{g}_2$ do not contain any loop, so that $p_1\leq \frac{v}{2}$ and $p_2\leq \frac{v}{2}$. Then, again by Lemma \ref{b111},
\begin{align*}
        p_n( \hat{g}_1,\hat{g}_2 ) & \leq  \sum_{(g_1,g_2)\in B^n_{\hat{g}_1,\hat{g}_2}}  \frac{1}{{\binom{n-p_1}{v-p_1}}(v-p_1)!}
        \frac{1}{{\binom{n-p_2}{v-p_2}}(v-p_2)!}
        \\
        & = 
        \frac{\textrm{card}(B^n_{\hat{g}_1,\hat{g}_2})}{{\binom{n-p_1}{v-p_1}}(v-p_1)!{\binom{n-p_2}{v-p_2}}(v-p_2)!}\\
        & \leq 
        \frac{{\binom{n-1}{v-1}} {v!}^2  }{{\binom{n-p_1}{v-p_1}}(v-p_1)!{\binom{n-p_2}{v-p_2}}(v-p_2)!}
        = O \left(n^{v-1-(v-p_1+v-p_2)}\right).\\
\end{align*}
Therefore, if $p_1<\frac{v}{2},$ as $p_1\leq \frac{v-1}{2}$ we have
\[   p_n( \hat{g}_1,\hat{g}_2 ) = O(n^{-\frac 3 2}).\]

% \begin{align*}      &\p((\hat{\mathcal{G}}^1_1(\sigma_n,\rho_n),\hat{\mathcal{G}}^1_2(\sigma_n,\rho_n))=(\hat{g}_1,\hat{g}_2))) 
%         \leq C_{\hat{g}_1,\hat{g}_2}n^{p_1+p_2-v-1}
%         \leq \frac{C_{\hat{g}_1,\hat{g}_2}}{n\sqrt{n}}. \end{align*}
The same  holds if  $p_2<\frac{v}{2}$ and the only remaining terms are the cases when $p_1= \frac{v}{2}=v_1$ and $p_2= \frac{v}{2}=v_1.$ In this case,  both graphs have
necessarily  connected components having two vertices.  By Lemma~\ref{l14}, we obtain that the only non trivial contribution comes from $\hat g_1 = \hat g_2 = \widehat{\mathcal T}_{v_1}.$
By Lemma~\ref{lem17}, we obtain  
\[
 \frac{\textrm{card}\big(B^n_{\widehat{\mathcal T}_{v_1},\widehat{\mathcal T}_{v_1}}\big)}{{\binom{n-p_1}{v-p_1}}(v-p_1)!{\binom{n-p_2}{v-p_2}}(v-p_2)!} \left(1-O\left(\frac{1}n\right)\right)    \leq p_n(\widehat{\mathcal T}_{v_1},\widehat{\mathcal T}_{v_1})\leq \frac{\textrm{card}\big(B^n_{\widehat{\mathcal T}_{v_1},\widehat{\mathcal T}_{v_1}}\big)}{{\binom{n-p_1}{v-p_1}}(v-p_1)!{\binom{n-p_2}{v-p_2}}(v-p_2)!}.
\]
Moreover, each element of $B^n_{\widehat{\mathcal T}_{v_1},\widehat{\mathcal T}_{v_1}}$ can be characterized by a choice of   $i^1_2,i^1_3,\dots i^1_{v_1},j^1_1,\dots j_{v_1}^1$  pairwise distincts in $\{2,3,\dots,n\},$ so that
$$\textrm{card}\big(B^n_{\widehat{\mathcal T}_{v_1},\widehat{\mathcal T}_{v_1}}\big)={\binom{n-1}{2v_1-1}} {(2v_1-1)!}.$$
Since $v=2p_1=2p_2=2v_1,$ we get that  \[p_n(\widehat{\mathcal T}_{v_1},\widehat{\mathcal T}_{v_1})=\frac{1+o(1)}{n}.\]
Summarizing all cases, we get that $C_{\hat g_1, \hat g_2} =0$ unless $\hat g_1= \hat g_2 =\widehat{\mathcal T}_{v_1},$ in which case 
$C_{\widehat{\mathcal T}_{v_1},\widehat{\mathcal T}_{v_1}}=1.$
\end{itemize}

%\end{itemize}
\end{proof}
\subsection{Proof of Theorem \ref{cor}}

The proof of  Theorem~\ref{cor} is similar to that of Proposition~\ref{prop}. Instead of studying $\mathcal{G}_i^1$, we study $\mathcal{G}_i^{\{1,2,\dots,k\}}$. We will prove using the same argument that only the event $\left\{\sigma,\rho; \forall i \in \{1,2\}, \mathcal{G}_i^{\{1,2,\dots,k\}}(\sigma,\rho) \in \cup_{p\geq 1}\mathcal{T}^n_p \right\}$ will contribute to the limit.

% The proof is similar to that of Lemma~\ref{ll15}. \color{red} est-ce qu'on utilise le lemme 13? peut-on mettre seulement le lemme 14 ? \color{black}
\begin{proof}[Proof of Theorem~\ref{cor} in the case $m=2$] Let $\bf{v}$=$(v_1,v_2,\dots v_k)$ be fixed.
If $\forall i \le k,  c_i(\sigma^{-1}\rho)=v_i$, then $$\mathcal{G}_1^{\{1,2,\dots,k\}}(\sigma,\rho),\mathcal{G}_2^{\{1,2,\dots,k\}}(\sigma,\rho) \in \bigcup_{p\leq \sum_{i=1}^k v_k} \hat{\mathbb{G}}_p.$$
Since $\bigcup_{p\leq \sum_{i=1}^k v_k} \hat{\mathbb{G}}_p$ is finite, it is sufficient  to prove that for any pair $\hat{g}_1,\hat{g}_2 \in \bigcup_{p\leq \sum_{i=1}^k v_k} \hat{\mathbb{G}}_p$ having the same number of non-isolated vertices, there exists a  constant $C_{\hat{g}_1,\hat{g}_2,\bf{v}}$  such that under the assumptions  of Theorem~{\ref{cor}}, 
\begin{align*}
    \p\left((\hat{\mathcal{G}}^{\{1,2,\dots,k\}}_1(\sigma_n,\rho_n), \hat{\mathcal{G}}^{\{1,2,\dots,k\}}_2(\sigma_n,\rho_n))=(\hat g_1,\hat{g}_2) \cap A_{\bf v}  \right)=\frac{C_{\hat{g}_1,\hat{g}_2, \bf v}}{n^k}+ o\left(\frac{1}{n^k}\right),
\end{align*}
where $A_{\bf v}:=\{\forall i\leq k,  c_i(\sigma_n^{-1}\rho_n)=v_i\}$.\\

Let $\hat{g}_1,\hat{g}_2 \in  \bigcup_{p\leq \sum_{i=1}^k v_k} \hat{\mathbb{G}}_p$ be two unlabeled graphs having respectively $p_1$ and $p_2$ connected components and   $v$  vertices. Let  $B^{n,\bf v}_{\hat{g}_1,\hat{g}_2}$ be the set of couples  $({g}_1,{g}_2)$ with $n$ vertices, having the same non-isolated vertices such that
\begin{itemize}
 \item[-] $1,2,\dots,k$ are   non-isolated vertices of both graphs, 
 \item[-] for $i \in \{1,2\}$, the 
equivalence class  of $g_i$ is  $\hat{g}_i,$ 
 \item[-]  there exists $\sigma,\rho$ such that for $i \in \{1,2\}$, ${\mathcal{G}}^{\{1,2,\dots k\}}_i(\sigma,\rho)=g_i$ and $c_i(\sigma^{-1}\rho)=v_i$.  
\end{itemize}

As before, we denote by 
\[ p_{n,{\bf v}}(\hat{g}_1,\hat{g}_2) :=  \p\left((\hat{\mathcal{G}}^{\{1,2,\dots,k\}}_1(\sigma_n,\rho_n), \hat{\mathcal{G}}^{\{1,2,\dots,k\}}_2(\sigma_n,\rho_n))=(\hat g_1,\hat{g}_2) \cap A_{\bf v} \right) \]
and we have 
\begin{align*}
   p_{n,{\bf v}}(\hat{g}_1,\hat{g}_2) & =    \sum_{(g_1,g_2)\in B^{n,\bf v}_{\hat{g}_1,\hat{g}_2}}
         \p((\mathcal{G}^{\{1,2,\dots,k\}}_1(\sigma_n,\rho_n),\mathcal{G}^{\{1,2,\dots,k\}}_2(\sigma_n,\rho_n))=(g_1,g_2)) \\
        & = \sum_{(g_1,g_2)\in B^{n,\bf v}_{\hat{g}_1,\hat{g}_2}}
        \p(\sigma_n\in \mathfrak{S}_{n,g_1},
        \rho_n\in \mathfrak{S}_{n,g_2}) = \sum_{(g_1,g_2)\in B^{n,\bf v}_{\hat{g}_1,\hat{g}_2}}
        \p(\sigma_n\in \mathfrak{S}_{n,g_1})\p(\rho_n\in \mathfrak{S}_{n,g_2}). 
\end{align*}
Starting from there, we distinguish different cases:

\begin{itemize}
\item Case 1:  $\hat{g}_1$ and $\hat{g}_2$ have respectively $f_1$ and $f_2$ loops i.e edges of type $(i,i)$ with $f_1+f_2 >0$.
Then $2p_1 -f_1 \leq {v}$ and $2p_2-f_2\leq v$. Consequently, by Lemmas \ref{b111} and \ref{lem:negligible},
\begin{align*}
      p_{n,{\bf v}}(\hat{g}_1,\hat{g}_2)
        & = 
        \frac{\textrm{card}(B^{n,\bf v }_{\hat{g}_1,\hat{g}_2})}{{\binom{n-p_1}{v-p_1}}(v-p_1)!{\binom{n-p_2}{v-p_2}}(v-p_2)!}o\left(n^\frac{-f_1-f_2}{2}\right) \\
        & \leq         \frac{{\binom{n-k}{v-k}} {v!}^2  o\left(n^\frac{-f_1-f_2}{2}\right)}{{\binom{n-p_1}{v-p_1}}(v-p_1)!{\binom{n-p_2}{v-p_2}}(v-p_2)!}= n^{v-k-(v-p_1+v-p_2)}o\left(n^\frac{-f_1-f_2}{2}\right) =o(n^{-k}) .
\end{align*}
   \item Case 2: $\hat{g}_1$ and $\hat{g}_2$ do not contain any loop.
Then $p_1\leq \frac{v}{2}$ and $p_2\leq \frac{v}{2}$. Consequently,
\begin{align*}
        p_{n,{\bf v}}(\hat{g}_1,\hat{g}_2) & \leq 
        \frac{\textrm{card}(B^{n,\bf v}_{\hat{g}_1,\hat{g}_2})}{{\binom{n-p_1}{v-p_1}}(v-p_1)!{\binom{n-p_2}{v-p_2}}(v-p_2)!}
        \\ &\leq \nonumber
        \frac{{\binom{n-k}{v-k}} {v!}^2 }{{\binom{n-p_1}{v-p_1}}(v-p_1)!{\binom{n-p_2}{v-p_2}}(v-p_2)!}
        \\ &\leq \nonumber C n^{v-k-(v-p_1+v-p_2)}.
\end{align*}
Therefore, if $p_1<\frac{v}{2}$ or  $ p_2<\frac{v}{2}$ then $ p_{n,{\bf v}}(\hat{g}_1,\hat{g}_2) =o(n^{-k})$.
The only remaining terms are the cases when $p_1= \frac{v}{2}$ and $p_2= \frac{v}{2}.$ In this case, both graphs have necessarily  only connected components having two vertices. Assume that  one of the two graphs has a cycle. Then, by Lemma~\ref{lem172}, we have
\begin{align*} 
         p_{n,{\bf v}}(\hat{g}_1,\hat{g}_2)
         &\leq   \sum_{(g_1,g_2)\in B^{n,\bf v}_{\hat{g}_1,\hat{g}_2}} \frac{(\p(c_1(\sigma_n)=2)+\p(c_1(\rho_n)=2))}{{\binom{n-p_1}{v-p_1}}(v-p_1)!{\binom{n-p_2}{v-p_2}}(v-p_2)!} 
         \\&\leq C (\p(c_1(\sigma_n)=2)+\p(c_1(\rho_n)=2)) n^{-k}.
\end{align*}
Under $H_4$, we have $\p(c_1(\sigma_n)=2)+\p(c_1(\rho_n)=2))=o(1)$ so that $ p_{n,{\bf v}}(\hat{g}_1,\hat{g}_2) = o(n^{-k})$ as soon as one of the graph has a cycle.

As before, the only non-trivial contributions come from the cases when 
$\hat{g}_1=\hat{g}_2=\widehat{\mathcal{T}}_{p}$ for some 
\\ $p\leq \sum_{i=1}^k v_i$ 
% \begin{align*}
%         q_{n,p, \bf v}:=&\p\left((\hat{\mathcal{G}}^{\{1,2,\dots,k\}}_1(\sigma_n,\rho_n), \hat{\mathcal{G}}^{\{1,2,\dots,k\}}_2(\sigma_n,\rho_n))=(\widehat{\mathcal{T}}_{p},\widehat{\mathcal{T}}_{p}) \cap A_{\bf v} \right) 
%         \\&=\sum_{(g_1,g_2)\in B^{n,\bf v}_{\widehat{\mathcal{T}}_{p},\widehat{\mathcal{T}}_{p}}}
%         \p(\sigma_n\in \mathfrak{S}_{n,g_1})
%         \p(\rho_n\in \mathfrak{S}_{n,g_2}),
%         \end{align*}
and by Lemma~\ref{lem17}, we obtain 
\begin{align*}
 \frac{\textrm{card}\big(B^{n,\bf v}_{\widehat{\mathcal{T}}_{p},\widehat{\mathcal{T}}_{p}}\big)}{{\binom{n-p_1}{v-p_1}}(v-p_1)!{\binom{n-p_2}{v-p_2}}(v-p_2)!} \left(1-O\left(\frac{1}n\right)\right)    \leq  p_{n, \bf v}\left(\widehat{\mathcal{T}}_{p},\widehat{\mathcal{T}}_{p}\right) \leq \frac{\textrm{card}\big(B^{n,\bf v}_{\widehat{\mathcal{T}}_{p},\widehat{\mathcal{T}}_{p}}\big)}{{\binom{n-p_1}{v-p_1}}(v-p_1)!{\binom{n-p_2}{v-p_2}}(v-p_2)!}.
\end{align*}
One can conclude since, for any $n \ge 2p,$  
$$\textrm{card}\big(B^{n,\bf v}_{\widehat{\mathcal{T}}_{p},\widehat{\mathcal{T}}_{p}}\big)=\textrm{card}\big(B^{2p,\bf v}_{\widehat{\mathcal{T}}_{p},\widehat{\mathcal{T}}_{p}}\big) {\binom{n-k}{2p-k}}$$
\end{itemize}
and consequently, for any $p \le \sum_{i=1}^k{v_i},$ 
\begin{align*}
    C_{\widehat{\mathcal{T}}_{p},\widehat{\mathcal{T}}_{p},\bf v } &= \frac{\textrm{card}\left(B^{2p,\bf v}_{\widehat{\mathcal{T}_p},\widehat{\mathcal{T}_p}}\right)}{(2p-k)!},
\end{align*}
and $C_{\hat g_1, \hat g_2,\bf v  } =0,$ as soon as $(\hat g_1, \hat g_2) \notin \left\{(\widehat{\mathcal{T}}_{p},\widehat{\mathcal{T}}_{p}), p \le \sum_{i=1}^k{v_i}\right\}.$ As the constants $C_{\hat g_1, \hat g_2,\bf v  }$ do not depend on the distributions of $\sigma_n$ and $\rho_n,$ this concludes the proof of Theorem \ref{cor} in the case of two permutations. 
\end{proof}
 To extend to $m>2$, we will proceed by induction on the number $m$ of permutations.
Our main argument is the following lemma. 
\begin{lemma}\label{lem:iteration}
Let $(\sigma^1_n)_{n \ge 1},(\sigma^2_n)_{n \ge 1}$ be two sequences of random permutations such that \linebreak for any $n \ge 1,$ $\sigma^1_n, \sigma^2_n \in \s$. Assume that
\begin{itemize}
    \item[-]  For any $n\geq 1$, $\sigma^1_n$ and  $\sigma^2_n$ are independent. 
    \item[-] For any $n\geq 1$ and $\ell \in \{1,2\},$ for any $\sigma \in \s,$
     \begin{align*}
         \sigma^{-1}\sigma_n^\ell\sigma\overset{d}{=}\sigma_n^\ell.
     \end{align*}
   \item[-] For any $k \geq 1$,  
    \begin{align*}
     \lim_{n\to\infty} \E\left(\left(\frac{\#_1 \,\sigma^{1}_n}{\sqrt n}\right)^k\right)=0 &\quad  \textrm{ and } \quad  \lim_{n\to\infty} \frac{\E(\#_2\,\sigma^1_n)}{n} =0.
    \end{align*}
    \end{itemize}

    Then,
    \begin{align} \label{eq:iteration}
    \lim_{n\to\infty} \E\left(\left(\frac{\#_1 (\sigma^{1}_n\sigma^{2}_n)}{\sqrt n}\right)^k\right)=0 &\quad  \textrm{ and } \quad  \lim_{n\to\infty} \frac{\E(\#_2(\sigma^1_n\sigma^{2}_n))}{n} =0. \quad \quad.
    \end{align}
\end{lemma}

\begin{proof}
We will only give a sketch of the proof. The idea is to repeat the same study as in the case $m=2$ in the two particular quantities.
\begin{itemize}
  \item  Take  $k \ge 1$ and $v_1=v_2=\dots=v_k=1.$  One can show that, under the hypotheses of Lemma \ref{lem:iteration},
    \begin{align*}
\lim_{n\to \infty}  \sum_{ \substack{  \hat{g}_i,\hat{g}'_i \in \hat{\mathbb{G}}_{1}, \, 1\leq i\leq k }}n^{\frac{k}{2}}
\p(\hat{\mathcal{G}}^{\{1,2,\dots,k\}}(\sigma^1_n,\sigma^2_n)=(\hat{g}_1,\hat{g}'_1,\hat{g}_2,\dots \hat{g}'_k ))=0.
\end{align*}
This leads to the first limit in \eqref{eq:iteration}.
    \item Take $k=1$ and $v_1=2.$  One can show  that, under the hypotheses of Lemma \ref{lem:iteration},
    \begin{align*}
 \forall \hat{g}_1,\hat{g}_2  \in \hat{\mathbb{G}}_{2}, \lim_{n\to \infty} \p((\hat{\mathcal{G}}^1_1(\sigma^1_n,\sigma^2_n),\hat{\mathcal{G}}^1_2(\sigma^1_n,\sigma^2_n))=(\hat{g}_1,\hat{g}_2)))=0.
\end{align*}
  This leads to the second limit in \eqref{eq:iteration}.
\end{itemize} 
\end{proof}

\section{Further discussion}
% \begin{rem} \label{rm1}
% One can see that the condition $H_1$ can be replaced by the following condition.
% 
% $\hat{H}_1 :$ For any $k_1,k_2\geq1$, for any $\varepsilon >0$, there exists $n_0$ such  that for any $n>n_0$, for any $g_1\in \mathbb{G}^n_{k_1}$, $g_2\in \mathbb{G}^n_{k_2}$,
% \begin{align*}
% (1-\varepsilon) \p(\sigma_n \in \mathfrak{S}_{n,g_1})\p(\rho_n \in \mathfrak{S}_{n,g_2})&\leq \p(\sigma_n \in \mathfrak{S}_{n,g_1},\rho_n \in \mathfrak{S}_{n,g_2})\leq (1+\varepsilon)\p(\sigma_n \in \mathfrak{S}_{n,g_1})\p(\rho_n \in \mathfrak{S}_{n,g_2}).  
% \end{align*}
% When both permutations are conjugation invariant, we don't need a uniform bound. One can have both bound for every finite graph and  by adding trivial edges to get higher dimensions. \color{red} difficile à comprendre \color{black}
% \end{rem}
% 
%  \color{red} à finir \color{black}
% 
% \begin{rem}\label{opt}
In this last section, we make a few remarks on the optimality of the assumptions $H_3$ and $H_4$ in Theorem \ref{cor}. We assume hereafter that $H_1$ and $H_2$ hold true and consider for the sake of clarity the case $m=2.$
\begin{itemize}
 \item The assumption  $H_3$ is optimal in the sense that if 
$$ \liminf_{n\to\infty} n^{-\frac k 2}\min(\E((\#_1\,\sigma_n)^k),\E((\#_1\,\rho_n)^k))=\varepsilon_k>0,$$ 
        then
        $$ \liminf_{n\to\infty} \E( (\#_1(\sigma_{n}\rho_n))^k) \geq \E(\xi_1^k)+\varepsilon^2_k.$$
Indeed, going back to the equation \eqref{st}, one can see that in the case $v_1=v_2=,\dots=v_k=1,$ if $\hat{g}$ is the class of  the graph with adjacency matrix
        $ {\rm  Id}_k$ the event $\{ (\hat{\mathcal{G}}^{1,2,\dots,k}_1(\sigma_n,\rho_n), \hat{\mathcal{G}}^{1,2,\dots,k}_2(\sigma_n,\rho_n))=(\hat{g},\hat{g})\}$ will contribute to the limit, leading to the term $\varepsilon^2_k.$
\item  Similarly  $H_4$ is optimal in the sense that if $$ \liminf_{n\to\infty} \left(\frac{\min(\E(\#_2\,\sigma_n),\E(\#_2\,\rho_n))}{n}\right)=\varepsilon^\prime >0,$$ 
%\end{rem}
then,
        $$ \liminf_{n\to\infty} \E\left( \left(\#_1(\sigma_{n}\rho_n)\right)^2\right) \geq 2+{\varepsilon'}^2.$$
Indeed, as above, in the case $v_1=v_2=1,$  if $\hat{g}^\prime$ is the class of the graph with adjacency matrix \linebreak
        $(\begin{smallmatrix}
0 & 1  \\ 
1 & 0 
\end{smallmatrix}),$ the event $\{ (\hat{\mathcal{G}}^{1,2,\dots,k}_1(\sigma_n,\rho_n), \hat{\mathcal{G}}^{1,2,\dots,k}_2(\sigma_n,\rho_n))=(\hat{g}^\prime,\hat{g}^\prime)\}$ will contribute to the limit. 
\item Assume now that one of the bounds in $H_3$ is not satisfied. More precisely, assume that there exists $k\geq 1$ such that
$$ \liminf_{n\to\infty} n^{-\frac{k}{2}} \E((\#_1\sigma_n)^k)=\varepsilon_k>0,$$ 
or 
$$\liminf_{n\to\infty} \frac{\E(\#_2\,\sigma_n)}{n}=\varepsilon'>0.$$
  Then, by similar arguments, one can check that  the convergences 
  $$\forall k\geq 1, \lim_{n\to\infty}n^{-\frac{k}{2}} \E((\#_1 \,\rho_n)^k)=0 \quad \textrm{ and }\quad \lim_{n\to\infty} \frac{\E(\#_2\,\rho_n)}{n}=0 $$
  are a necessary condition to obtain \eqref{maineq} and that the convergences 
    $$ \forall k\geq 1, \lim_{n\to\infty}n^{-\frac{k}{2}} \E((\#_1\,\rho_n)^{k})=0, \ \ 
            \limsup_{n\to\infty} n^{-\frac{k}{2}} \E((\#_1\,\sigma_n)^k) <\infty
    \quad \textrm{ and }\quad \lim_{n\to\infty} \frac{\E(\#_2\,\rho_n)}{n}=0$$
%    \color{red} petit doute ici : est-ce la bonne normalisation sur l'equation du milieu ? \color{black}
      are a sufficient condition to obtain \eqref{maineq}.
\end{itemize}

\bibliography{bib}
\end{document}